\documentclass[oneside,english]{amsart}
\usepackage[T1]{fontenc}
\usepackage[latin9]{inputenc}
\usepackage{geometry}
\usepackage{verbatim}
\usepackage{amstext}
\usepackage{amsthm}
\usepackage{hyperref}

\hypersetup{
   colorlinks = true,
   urlcolor = blue,
   linkcolor = red
}

\makeatletter
\numberwithin{equation}{section}
\numberwithin{figure}{section}
\theoremstyle{plain}
\newtheorem{thm}{\protect\theoremname}
  \theoremstyle{plain}
  \newtheorem{lem}[thm]{\protect\lemmaname}
  \theoremstyle{remark}
  
  \theoremstyle{plain}
  \newtheorem{cor}[thm]{\protect\corollaryname}
\makeatother

\usepackage{babel}
  \providecommand{\lemmaname}{Lemma}
  \providecommand{\remarkname}{Remark}
  \providecommand{\corollaryname}{Corollary}
\providecommand{\theoremname}{Theorem}

\begin{document}

\title{A Generalized Newton-Girard Identity}

\author{Tanay Wakhare}

\address{University of Maryland, College Park, MD 20742, USA}
\email{twakhare@gmail.com}

\begin{abstract}We present a generalization of the Newton-Girard identities, along with some applications. As an addendum, we collect many evaluations of symmetric polynomials to which these identities apply.
\end{abstract}
\maketitle

\section{Introduction}

The theory of symmetric polynomials has been explored by mathematicians for centuries, and is intimately connected to many different fields, including combinatorial enumeration and the representation theory of the symmetric group. Section \ref{sec4} contains many examples of special symmetric polynomials, such as binomial coefficients, $q$-binomial coefficients, and Stirling numbers, all of which have been extensively and independently studied. Two of the fundamental identities for symmetric polynomials are the \textit{Newton-Girard} identities, which state that \cite[(Chapter 1, 2.10-2.11)]{Macdonald}
\begin{equation}\label{newtone}
ne_n = \sum_{k=1}^n(-1)^{k-1}p_ke_{n-k}
\end{equation}
and
\begin{equation}\label{newtonh}
nh_n = \sum_{k=1}^n p_k h_{n-k}.
\end{equation}
Here, $e_n$ are the \textit{elementary symmetric functions}, $h_n$ are the \textit{complete symmetric functions}, and $p_n$ are the \textit{power sums}. They form generating sets for the ring of symmetric functions, and satisfy the following definitions. Letting $\{x_i\} = \{x_1, x_2, \ldots\}$ denote a (possibly infinite) set of variables, we have
\begin{align}
e_n  &:= \sum_{i_1<i_2<\ldots < i_n} x_{i_1} x_{i_2} \cdots x_{i_n} \\
h_n  &:= \sum_{i_1\leq i_2\leq\ldots \leq i_n} x_{i_1} x_{i_2} \cdots x_{i_n} \\
p_n  &:= \sum_{i=1}^\infty x_i^n. 
\end{align}
We require $e_0=h_0=1$ and $p_0=0$. Then we have the following formal generating series and products \cite{Macdonald2}:
\begin{align}
E(t) &= \sum_{k=0}^{\infty}e_k t^k = \prod_{i=1}^{\infty}\left(1+tx_i\right), \label{eprod} \\
H(t) &= \sum_{k=0}^{\infty}h_k t^k = \prod_{i=1}^{\infty}\left(1-tx_i\right)^{-1} = E(-t)^{-1}, \label{hprod} \\
P(t) &= \sum_{k=1}^{\infty}p_k t^{k-1} = \sum_{i=1}^{\infty}\frac{x_i}{1-tx_i} = \frac{H'(t)}{H(t)} = \frac{E'(-t)}{E(-t)}\label{pprod}.
\end{align}
We also introduce some closely related bases for the ring of symmetric functions. Let $\lambda \vdash n$ denote a partition of $n$, so that $\lambda = (\lambda_1, \lambda_2,\cdots)$ with $n = \sum_i \lambda_i$ and $\lambda_1 \geq \lambda_2 \geq \cdots$. Then we define $e_\lambda = e_{\lambda_1}e_{\lambda_2}\cdots$. Similarly, we have $h_\lambda = h_{\lambda_1}h_{\lambda_2}\cdots$ and $p_\lambda = p_{\lambda_1}p_{\lambda_2}\cdots$. Finally, we define the monomial symmetric functions $m_\lambda:=\sum_{\sigma \in S_\lambda} x_{\sigma(1)}^{\lambda_1} x_{\sigma(2)}^{\lambda_2} \cdots$, where the sum is over the set of permutations giving distinct terms in the sum (so that the coefficient of any monomial in the sum is simply $1$). When $\lambda=(1,1,\cdots)$, $m_\lambda$ reduces to an elementary symmetric function.

The derivation of the Newton-Girard identities from these generating products is instructive. There are many different proofs, using everything from recursive approaches to the Cayley-Hamilton Theorem \cite{Baker1} \cite{Mead1} \cite{Kalman1}. One of the easiest proofs, detailed by Macdonald \cite[Chapter 1]{Macdonald}, formally manipulates the generating products for $E(t)$ and $H(t)$. We begin with $E(t)$ and apply the differential operator $t \frac{d}{dt}$. We can take a logarithmic derivative of the product, yielding $$t \frac{d}{dt} E(t)=  \sum_{k=1}^{\infty} k e_k t^k = \prod_{i=1}^{\infty}\left(1+tx_i\right)  \sum_{i=1}^\infty \frac{tx_i}{1+tx_i}.$$ We recognize the sum as $tP(-t)$, which then gives $$ \sum_{k=1}^{\infty} k e_k t^k =\left( \sum_{k=0}^\infty e_k t^k \right)\left(\sum_{k=1}^\infty   p_k (-1)^{k+1} t^k  \right) = \sum_{ k=1}^\infty t^k \sum_{j=1}^k  e_{k-j} p_j (-1)^{j+1} .$$ Equating coefficients of $t^k$ proves Equation \eqref{newtone}. Equation \eqref{newtonh} is proved analogously, by applying the operator $t\frac{d}{dt}$ to the series and product forms for $H(t)$.

We explore this method in the rest of this paper: taking derivatives and logarithmic derivatives while switching between the product and series forms of $E$, $P$, and $H$. In Section \ref{sec2} we present a Newton type representation for the power sum $p_n$. In Section \ref{sec3} we present the main result of our study: a generalization of the Newton-Girard identities to two arbitrary sets of variables, given in Theorem \ref{thm4}. Finally, in Section \ref{sec4} we collect some examples of symmetric polynomial evaluations.

%%%%%%%%%%%%%%%%%%%%%%%%

\section{Newton-type representation for $p_n$}\label{sec2}
While Identities \eqref{newtone} and \eqref{newtonh} have been extensively studied in the literature, in Theorem \ref{thm1} we present a similar expression for $p_n$ for the first time. We note the \textit{extreme similarity} to \eqref{newtone} and \eqref{newtonh}, except for the factor of $k$. The first equality is given by \cite[Lemma 2.1]{Merca1}, while the second appears to be new.
\begin{thm}\label{thm1}
We have the Newton-type convolutions
\begin{equation}
p_n = \sum_{k=0}^n (-1)^{k-1} k e_k h_{n-k} = \sum_{k=0}^n (-1)^{n-k} k h_k e_{n-k}.
\end{equation}
\end{thm}
\begin{proof}
We start with the product form $$\sum_{k=0}^\infty p_k t^k =  tP(t) = t\frac{H'(t)}{H(t)} = t\frac{E'(-t)}{E(-t)}.$$ By looking at the definition of $E$ and $H$, we also observe $E(-t) = \prod_{i=1}^\infty  (1-x_it) = H(t)^{-1}$. This means $$tP(t) = tH'(t)E(-t) = tE'(-t)H(t).$$ Taking products of both sides and equating coefficients of $t^n$ completes the proof:

$$  tH'(t)E(-t) = \left(\sum_{k=0}^\infty kh_k t^{k} \right)\left(\sum_{k=0}^\infty e_k t^{k} (-1)^k \right)= \sum_{n=0}^\infty t^n\sum_{k=0}^n  (-1)^{n-k} k h_k e_{n-k}  $$ and $$  tE'(-t)H(t) =\left( \sum_{k=0}^\infty ke_k t^{k} (-1)^{k-1}\right)\left(\sum_{k=0}^\infty h_k t^{k} \right) = \sum_{n=0}^\infty t^n \sum_{k=0}^n (-1)^{k-1} k e_k h_{n-k}.  $$
\end{proof}

%%%%%%%%%%%%%%%%%%%%%%%%%%%%%%%%%%%%%%5

\section{Two Variable Newton Identities}\label{sec3}
We now extend the formulae \eqref{newtone} and \eqref{newtonh} to two (possibly infinite) sets of variables, motivated by the result \cite[(6.3)]{Macdonald2}
\begin{equation}
\prod_{i,j}(1-x_i y_j)^{-1} = \sum_{\lambda} h_\lambda(x) m_\lambda(y).
\end{equation}
In other words, $m$ and $h$ are \textit{dual bases} for the ring of symmetric functions. This motivates us to consider 
\begin{align}
\Pi\left(t\right):&= \prod_{i,j}(1-x_i y_j t)^{-1},\label{bigPi}\\
\pi\left(t\right):&= \prod_{i,j}(1+x_i y_j t).\label{smallPi}
\end{align}

In fact, these are generating products for certain symmetric polynomials. This lemma allows us to switch between the product and series forms of $\Pi(t)$, which simplifies the following analysis.
\begin{lem}\label{lem2}
We have the generating products
\begin{equation}
\Pi(t) =  \sum_{k=0}^\infty t^k \sum_{\lambda \vdash k} h_\lambda (x) m_\lambda(y),
\end{equation}
and
\begin{equation}
\pi(t) =  \sum_{k=0}^\infty t^k \sum_{\lambda \vdash k} e_\lambda (x) m_\lambda(y).
\end{equation}
\end{lem}
\begin{proof}
We can expand $\Pi(t)$ as the generating product for $h_n(x)$ using \eqref{hprod} and then factor out the contributions from $\{y_j\}$, yielding
\begin{align*}
\Pi(t) = &\prod_{j=1}^\infty \left( \prod_{i=1}^\infty  \frac{1}{1-x_iy_j t} \right) \\
=  &\prod_{j=1}^\infty \left( \sum_{k=0}^\infty h_k(x) y_j^k t^k\right)  \\
 = &\left(  1+h_1(x) y_1 t+ h_2(x) y_1^2 t^2 +h_3(x) y_1^3 t^3 + h_4(x) y_1^4 t^4 + \ldots     \right) \\
 &\times \left(  1+h_1(x) y_2 t+ h_2(x) y_2^2 t^2 +h_3(x) y_2^3 t^3 + h_4(x) y_2^4 t^4 + \ldots     \right) \\
 &\times \left(  1+h_1(x) y_3 t+ h_2(x) y_3^2 t^2 +h_3(x) y_3^3 t^3 + h_4(x) y_3^4 t^4 + \ldots     \right) \\
&\times\cdots \\
 = &1 +t\left( h_1(x) \sum_{i=1}^\infty y_i \right)  + t^2 \left(  h_2(x)  \sum_{i=1}^\infty y_i^2 + h_1(x)h_1(x)  \sum_{i\neq j}^\infty y_i y_j    \right) + \cdots\\
= & \sum_{k=0}^\infty t^k \sum_{\lambda \vdash k} h_\lambda (x) m_\lambda(y).
\end{align*}
We now explain the last equality. For what follows, let $\lambda = (1^{n_1} 2^{n_2} \ldots k^{n_k})$ be a partition of $k$ where $1$ appears $n_1$ times, $2$ appears $n_2$ times, and so on. First, note that the coefficient of $t^k$ will be a polynomial in $\{h_i(x)\}$ and $\{y_i\}$. A general term in this polynomial will be $h_{n_1}(x)h_{n_2}(x)\cdots h_{n_k}(x)$ times $y_{j_1}^{n_1} \cdots y_{j_k}^{n_k}$. When we collect $t^k$ terms, we then sum over \textit{all possible distinct permutations} of the indices $\{j_1, \ldots, j_k\}$. This means that while considering terms corresponding to a particular partition $\lambda$, we can then factor out $h_\lambda$ and then rewrite the resulting sum over monomials in the $\{y_i\}$ as $m_\lambda(y)$. Every partition $\lambda$ of $k$ will have $h_\lambda$ and $m_\lambda$ terms of this form corresponding to it, leading to the given coefficient. 

%We have to make three logical leaps: that we should index the terms in the coefficient of $t^k$ by partitions, then that we can factor out the contribution of each $h_\lambda$, and finally that we can sum large classes of $y$ terms to recover $m_\lambda$.

The second identity is proven with the exact same reasoning:
\begin{align*}
\pi(t) = &\prod_{j=1}^\infty \left( \prod_{i=1}^\infty  ({1+x_iy_j t}) \right) \\
=  &\prod_{j=1}^\infty \left( \sum_{k=0}^\infty e_k(x) y_j^k t^k\right)  \\
 = &1 +t\left( e_1(x) \sum_{i=1}^\infty y_i \right)  + t^2 \left(  e_2(x)  \sum_{i=1}^\infty y_i^2 + e_1(x)e_1(x)  \sum_{i\neq j}^\infty y_i y_j    \right) + \cdots\\
= & \sum_{k=0}^\infty t^k \sum_{\lambda \vdash k} e_\lambda (x) m_\lambda(y).
\end{align*}
\end{proof}
When $y_1 =1$ and $y_j=0, j \geq 2,$ this product identity specializes to the generating products for $E(t)$ and $H(t)$. For instance, the sum $\sum_{\lambda \vdash k}h_\lambda(x)m_\lambda(y)$ reduces to $h_k(x)$ since $m_\lambda(y)$ is only nonzero when $\lambda$ consists of a single part, $k$.

\begin{cor}
We have the symmetries 
\begin{align}
\sum_{\lambda \vdash k} h_\lambda (x) m_\lambda(y) &= \sum_{\lambda \vdash k} h_\lambda (y) m_\lambda(x) \\
\sum_{\lambda \vdash k} e_\lambda (x) m_\lambda(y) &= \sum_{\lambda \vdash k} e_\lambda (y) m_\lambda(x).
\end{align}
\end{cor}
\begin{proof}
The corollary follows from the symmetry of $\Pi$ and $\pi$ with respect to $x$ and $y$, visible in the definitions \eqref{bigPi} and \eqref{smallPi}. 
\end{proof}
Using this product expansion, we can then present a generazation of the Newon-Girard identities.
\begin{thm}\label{thm4}
We have the generalized Newton-Girard identities
\begin{align}
n \sum_{\lambda \vdash n} h_\lambda (x) m_\lambda(y) &= \sum_{k=0}^n p_{n-k}(x)p_{n-k}(y) \sum_{\lambda \vdash k}   h_\lambda (x) m_\lambda(y), \\
n \sum_{\lambda \vdash n} e_\lambda (x) m_\lambda(y) &= \sum_{k=0}^n (-1)^{n-k}p_{n-k}(x)p_{n-k}(y) \sum_{\lambda \vdash k}   e_\lambda (x) m_\lambda(y).
\end{align}
\end{thm}

\begin{proof}
We can repeat the original proof of the Newton-Girard identities for the generating products above, with some minor modifications. We consider $\Pi(t)$ and $\pi(t)$, then apply the differential operator $t\frac{d}{dt}$ and equate coefficients of $t^k$. We begin by considering $\Pi(t)$ and taking a logarithmic derivative: $$t \frac{d}{dt}\Pi(t) = \Pi(t) \sum_{i,j=1}^\infty \frac{x_iy_jt}{1-x_iy_jt}.$$ We recognize $$ \sum_{i,j=1}^\infty \frac{x_iy_jt}{1-x_iy_j t} = \sum_{i,j=1}^\infty \sum_{k=1}^\infty{x_i^k y_j^k t^k} =\sum_{k=1}^\infty t^k \left(\sum_{i=1}^\infty x_i^k\right)\left(\sum_{j=1}^\infty y_j^k\right) = \sum_{k=1}^\infty t^k p_k(x)p_k(y).$$  
Since by Lemma \ref{lem2} $$\prod_{i,j}(1-x_i y_jt)^{-1} =\sum_{k=0}^\infty t^k\sum_{\lambda \vdash k} h_\lambda(x) m_\lambda(y),$$ we equate coefficients of $t^k$ in $$  t \frac{d}{dt}\Pi(t) = \sum_{k=0}^\infty t^k   k \sum_{\lambda \vdash k} h_\lambda (x) m_\lambda(y)$$ and $$   \Pi(t) \sum_{i,j=1}^\infty \frac{x_iy_jt}{1-x_iy_jt}   =\left( \sum_{k=0}^\infty t^k\sum_{\lambda \vdash k} h_\lambda(x) m_\lambda(y) \right)  \left( \sum_{k=0}^\infty t^k p_k(x)p_k(y) \right).$$ The same process, applied to $\pi(t)$, proves the second generalized Newton-Girard identity.
\end{proof}

%There are some very fundamental follow up questions about these identities. This is because the original Newton identities are defined for the infinite set of variables $\{x_1, x_2, \ldots\}$, while our result is for the set of infinite variables $\{x_1, x_2, \ldots, y_1, y_2, \ldots\}$. \textit{We can create a bijection between both, since both sets are countably infinite}. That means there should be a way to label the original $x_i$ so that we recover Theorem \ref{thm4}. It is not immediately apparent just what our new identities tell us, if they even tell us something new. 

%However, it appears very hard to derive Theorem \ref{thm4} from the Newton identities \eqref{newtone} and \eqref{newtonh}. 

While the classical Newton identities deal with a single set of variables, introducing a second set of variables in $\Pi(t)$ and $\pi(t)$ allowed us to prove analogs of the Newton identities. The natural next step would be to extend our results to three or more sets of variables, and see if this reveals new structure in the Newton identities. However, the three variable extension does not appear to have a similar closed form expression.

\section{Evaluation of symmetric polynomials}\label{sec4}
As an addendum, we collect several evaluations of symmetric polynomials. We can formulaically apply any of the results above, simply substituting the objects below whenever $e_k$, $h_k$, and $p_k$ appear. By applying Theorem \ref{thm4} to any of the specialized sets of variables $\{x_i\}$, we obtain identities of independent interest.

%These are several cases we are liable to see often, like constants and arithmetic and geometric progressions. 

Firstly, if we set each $x_i$ to $1$, we recover binomial coefficients since $e_k=\binom{n}{k}$  and $h_k=\binom{n+k+1}{k}$. These follow from counting the number of unique orderings $i_1<\cdots<i_k$ and $i_1\leq \cdots \leq i_k$, with $1\leq i_l \leq n$. When we consider geometric progressions and set $x_i=q^i$, we obtain $q$-binomial coefficients as a consequence of the $q$-binomial theorem. The \textit{$q$-binomial coefficient} is defined, for $n \geq k$, as 
\begin{equation}
\begin{bmatrix} n \\ k \end{bmatrix}_q:= \frac{(q;q)_n}{(q;q)_k(q;q)_{n-k}},
\end{equation}
where 
\begin{equation}
(a;q)_n  := 
\begin{cases}
(1-a)(1-aq)\ldots (1-aq^{n-1}),   &n \geq 1 \\
1 &n=0
\end{cases}
\end{equation}
is a \textit{q-Pochhammer} symbol. We rely on the $q$-binomial theorem \cite[Chapter 1]{Gasper}, namely
\begin{equation}
(-t;q)_n = \prod_{k=0}^{n-1} (1+q^kt) = \sum_{k=0}^n q^{\binom{k}{2}} \begin{bmatrix} n \\ k \end{bmatrix}_q t^k,
\end{equation}
and
\begin{equation}
\frac{1}{(t;q)_n} = \prod_{k=0}^{n-1}\frac{1}{ (1-q^kt)} = \sum_{k=0}^\infty \begin{bmatrix} n+k-1 \\ k \end{bmatrix}_q t^k.
\end{equation}
Recognizing $(-t;q)_n$ as the generating product for elementary symmetric functions, and $\frac{1}{(t;q)_n}$ as generating complete symmetric functions while comparing coefficients of $t^k$ yields the given evaluation. 

Next we consider the \textit{Jacobi-Stirling} numbers, a generalization of the Stirling numbers which naturally arise in the spectral theory of the Jacobi differential expression \cite{Everitt1}. We let $\begin{bmatrix} n \\ k \end{bmatrix}_\gamma$ denote a Jacobi-Stirling number of the first kind, and  $\begin{Bmatrix} n \\ k \end{Bmatrix}_\gamma$ a Jacobi-Stirling number of the second kind. The article \cite{Mongelli1} proves the equalities 
\begin{equation}\label{stir1}
e_k(2\gamma, 2+4\gamma, \ldots,n(n-1+2\gamma)) = \begin{bmatrix} n \\ k \end{bmatrix}_\gamma
\end{equation}
and 
\begin{equation}\label{stir2}
h_k(2\gamma, 2+4\gamma, \ldots,n(n-1+2\gamma)) = \begin{Bmatrix} n \\ k \end{Bmatrix}_\gamma
\end{equation}
by showing that both sides of Identities \eqref{stir1} and \eqref{stir2} satisfy the same recurrence relations. This means that we can apply any of our theorems about $e_k$ and $h_k$ to the Jacobi-Stirling numbers, which is the main idea of \cite{Merca2} and \cite{Merca3}.
%%TODO p_k becomes evaluations of sum of powers of a quadratic

We now consider $w_{m,r}(n,k)$, an $r$-Whitney number of the first kind, and $W_{m,r}(n,k)$, an $r$-Whitney number of the second kind. These are common generalizations of the Stirling and Whitney numbers. Results about them are systematized in the work \cite{Merca4}, which explores them as the symmetric function evaluations
\begin{equation}
e_k(r,r+m,r+2m,\ldots,r+nm) =(-1)^k w_{m,r}(n+1,n+1-k)
\end{equation}
and 
\begin{equation}
h_k(r,r+m,r+2m,\ldots,r+nm) = W_{m,r}(n+k,n).
\end{equation}
Sums of powers of arithmetic progressions have been explored several times. We have the result \cite{Merca4}
\begin{equation}
p_k(r,r+m,r+2m,\ldots,r+nm)  = \sum_{j=0}^{n} (r+jm)^k = \frac{m^k}{k+1}\left(B_{k+1}\left(l+1+\frac{r}{m}\right)-B_{k+1}\left(\frac{r}{m}\right)\right),
\end{equation}
where $B_n(x)$ is a Bernoulli polynomial.

We now move from considering combinatorial quantities to multiple zeta values, which have been the subject of increasing study by number theorists. We define 
\begin{equation}
\zeta(s_1,\ldots,s_r) := \sum_{n_1 > n_2 > \cdots > n_r \geq 1} \frac{1}{{n_1}^{s_1}\cdots {n_r}^{s_r}}
\end{equation}
 and 
\begin{equation}
\zeta^\ast(s_1,\ldots,s_r) := \sum_{n_1 \geq n_2 \geq \cdots \geq n_r \geq 1} \frac{1}{{n_1}^{s_1}\cdots {n_r}^{s_r}},
\end{equation}
as \textit{multiple zeta} and \textit{multiple zeta star} values. These are generalizations of the Riemann zeta function which are increasingly important in fields from conformal field theory to knot theory. Letting $x_i = \frac{1}{i^s}$ be an infinite set of variables, we see that the definitions of these multiple zeta values and our symmetric polynomials $e_k$ and $h_k$ coincide. We also recognize $p_k$ as the Riemann zeta function since $p_k = \sum_{i=1}^\infty \frac{1}{i^{sk}}=\zeta(sk)$. This point of view is extensively explored in \cite{Wakhare1}.

Our last result naturally extends the multiple-zeta point of view. We take $x_i = \frac{1}{l_i^s}$, where $l_i$ is the $i$-th prime. We then see that $e_k$ is a summation ranging over all \textit{squarefree} integers with $k$ distinct prime factors, while $h_k$ is a summation over \textit{all} integers with $k$ distinct prime factors. This is encoded by the sums $$\sum_{\substack{n\geq 1\\\omega(n)=k}} \frac{\mu(n)^2}{n^s}$$ and $$\sum_{\substack{n\geq 1\\\omega(n)=k}} \frac{1}{n^s},$$
since $\mu(n)^2$, with $\mu(n)$ the \textit{M\"obius function}, is the indicator function for squarefree numbers. Additionally, $\omega(n)$ denotes the number of distinct prime factors of $n$. We then have an evaluation of $p_k$ in terms of the \textit{prime zeta function} $P(s)$, defined as $\sum_{i=1}^\infty \frac{1}{l_i^s}$. The reader is referred to \cite{Wakhare2} for further details.

We note some overarching trends: letting $x_i$ be a natural number and then considering an \textit{infinite number of variables $\{x_1,x_2,\ldots\}$} gives us information about multiple zeta values. Letting $x_i$ be a prime number gives us information about natural numbers. We have studied symmetric functions where $x_i$ is a polynomial of degree $0$, $1$, and $2$ in $i$.  Considering a \textit{finite number of variables} gives us information about classical combinatorial objects. Are there some results that hold for the symmetric functions of $x_i$ when $x_i$ is a polynomial of arbitrary degree in $i$? We also note that many $q$-series identities follow from setting $x_i$ to be some function of $q_i$, then taking $t\to1$ and reindexing any summations to be over powers of $q$.

These examples highlight the vast reach of symmetric polynomial identities, and the fact that any progress on basic symmetric function identities will have many varied applications.

\bigskip

%%%how to make this bigger?
\begin{center}

  \begin{tabular}{  c| c |c | c }
    $x_i$ & $e_k$ & $h_k$ & $p_k$\\ 
    \hline
    $\{1,1,\ldots,1\}$ ($n$ times) & $\binom{n}{k}$ & $\binom{n+k-1}{k}$ & $n$\\ 
    $\{1,q,q^2, \ldots, q^{n-1}\}$ & $q^{\binom{k}{2}} \begin{bmatrix} n \\ k \end{bmatrix}_q $ &  $ \begin{bmatrix} n+k-1 \\ k \end{bmatrix}_q$  & $ \frac{1-q^{nk}}{1-q^k}$\\ 
    $\{r,r+m,r+2m,\ldots,r+nm\}$ & $(-1)^k w_{m,r}(n+1,n+1-k)$ &  $ W_{m,r}(n+k,n)$  & Bernoulli polynomials\\ 
    $\{2\gamma, 2+4\gamma, \ldots,n(n-1+2\gamma)\}$ & $\begin{bmatrix} n \\ k \end{bmatrix}_\gamma$ &  $ \begin{Bmatrix} n \\ k \end{Bmatrix}_\gamma$  &  ??? \\ 
    $\{\frac{1}{1^s},\frac{1}{2^s}, \frac{1}{3^s}\ldots\}$ & $\zeta(\{s\}^k)$ & $\zeta^\ast(\{s\}^k)$ & $\zeta(sk)$\\ 
    $\{\frac{1}{2^s},\frac{1}{3^s}, \frac{1}{5^s},\ldots\}$ & $\sum_{\substack{n\geq 1\\\omega(n)=k}} \frac{\mu(n)^2}{n^s}$ &  $\sum_{\substack{n\geq 1\\\omega(n)=k}} \frac{1}{n^s}$  & $P(sk)$  \\ 
  \end{tabular}

\end{center}

%%justify studying newton girard
%%introduce m (monomial symmetric func)
%%introduce p_n(x) notation
%%prove finite result and cite infinite case as the limit n \to \infty
%%find some easy applications

\end{document}